\documentclass[12pt,english]{article}
\usepackage{mathptmx}

\usepackage[T1]{fontenc}
\usepackage[latin9]{inputenc}
\usepackage{geometry}
\usepackage{float}
\usepackage{amsmath}
\usepackage{amsthm}
\usepackage{amssymb}

\makeatletter

\providecommand{\tabularnewline}{\\}

\theoremstyle{plain}
\newtheorem{thm}{\protect\theoremname}[section]
\theoremstyle{definition}
\newtheorem{defn}[thm]{\protect\definitionname}
\theoremstyle{plain}
\newtheorem{cor}[thm]{\protect\corollaryname}
\theoremstyle{plain}
\newtheorem{lem}[thm]{\protect\lemmaname}
\ifx\proof\undefined
\newenvironment{proof}[1][\protect\proofname]{\par
	\normalfont\topsep6\p@\@plus6\p@\relax
	\trivlist
	\itemindent\parindent
	\item[\hskip\labelsep\scshape #1]\ignorespaces
}{%
	\endtrivlist\@endpefalse
}
\providecommand{\proofname}{Proof}
\fi
\theoremstyle{plain}
\newtheorem{prop}[thm]{\protect\propositionname}
\theoremstyle{definition}
\newtheorem{example}[thm]{\protect\examplename}

\usepackage{multicol}
\usepackage{changepage}

\makeatother

\usepackage{babel}
\providecommand{\corollaryname}{Corollary}
\providecommand{\definitionname}{Definition}
\providecommand{\examplename}{Example}
\providecommand{\lemmaname}{Lemma}
\providecommand{\propositionname}{Proposition}
\providecommand{\theoremname}{Theorem}

\begin{document}
\begin{adjustwidth*}{+2cm}{+2cm}

\begin{center}

\textbf{FINITE DIMENSIONAL NILPOTENT LEIBNIZ ALGEBRAS WITH ISOMORPHIC
MAXIMAL SUBALGEBRAS}

\end{center}

\

\

\begin{center}

Lindsey Farris

\end{center}

\

\begin{center}

North Carolina State University

\end{center}

\

\

\

\begin{center}

\textbf{ABSTRACT}

\end{center}

\

\begin{center}

\noindent Nilpotent Leibniz algebras with isomorphic maximal subalgebras
are considered. The algebras are classified for coclass zero, one,
and two. The results are field dependent. 

\end{center}

\end{adjustwidth*}

\section{Introduction}

Due to the huge number of nilpotent Leibniz algebras, classifying
them is done by placing further conditions on them. In this paper,
we consider nilpotent Leibniz algebras with isomorphic maximal subalgebras,
and classify them with respect to coclass. This approach was first
used by P\'eter Z. Hermann in group theory (\cite{Hermann}). These
results were later refined by Avinoam Mann (\cite{MannGpThry}). Karen
Holmes later extended the results to Lie algebras (\cite{Karenpaper}). 

The results contained in this paper hold over the complex numbers,
and at times, the results are broader. We note when the results are
restricted to $\mathbb{C}$. Throughout this paper, we refer to two
properties labeled as P1 and P2. Property P1 denotes that all maximal
subalgebras are isomorphic, while P2 refers to the property that,
for any maximal subalgebra $M$, the $dim(Z_{i}(M))$ depends only
on $i$, and not $M$. Note that P1 is a stronger condition, and so
P1 implies P2. Hermann only made use of P1, while Holmes introduced
P2 for some results as it was easier to work with. We follow Holmes'
lead on when to use P2. 

\section{Nilpotent Leibniz Algebras}
\begin{defn}
Let $A$ be a vector space over a field $\mathbb{F}$. Then $A$ is
a left Leibniz algebra if it is equipped with a bilinear map,
\[
\left[,\right]:A\times A\longrightarrow A
\]
which satisfies 
\begin{equation}
\left[a,\left[b,c\right]\right]=\left[\left[a,b\right],c\right]+\left[b,\left[a,c\right]\right].\label{eq:leibidentity}
\end{equation}
\end{defn}
Note that the bilinear map is often referred to as a multiplication,
and (\ref{eq:leibidentity}) is called the Leibniz identity. This
paper will refer to left Leibniz algebras simply as Leibniz algebras,
which will be denoted as $A$. Any other definitions or terms not
defined in this paper shall be the same as in (\cite{stitz}). In
order to define coclass, we begin by defining a nilpotent algebra
and its class. There are numerous definitions for nilpotent.
\begin{defn}
Let $A$ be a Leibniz algebra. We say that $A$ is nilpotent of class
$c$ if every product of $c+1$ elements is zero, and there is some
product of $c$ elements that is not zero. We will denote this by
$cl\left(A\right)$.
\end{defn}
\begin{defn}
Given a Leibniz algebra $A$ we can define the lower central series
to be 
\[
A=A^{1}\supseteq A^{2}\supseteq\cdots
\]
where the $A^{i}$ are ideals given by $A^{i+1}=\left[A,A^{i}\right]$.
Note that $A$ need not be nilpotent to define this series.
\end{defn}
\begin{cor}
(\cite{stitz}, Corollary 4.3) The Leibniz algebra $A$ is nilpotent
of class $c$ if $A^{c+1}=0$ but $A^{c}\neq0$.
\end{cor}
\begin{defn}
Suppose $A$ is nilpotent of class $c$. The upper central series
is given by 
\[
0=Z_{0}\left(A\right)\subseteq Z_{1}\left(A\right)\subseteq\cdots\subseteq Z_{c}\left(A\right)=A
\]
where $Z_{i}(A)$ is the largest subalgebra of $A$ such that $[Z_{i}(A),A]\subseteq Z_{i-1}(A)$
and $[A,Z_{i}(A)]\subseteq Z_{i-1}(A)$ for any $i\leq c$. Alternatively,
$Z_{i}(A)/Z_{i-1}(A)=Z(A/Z_{i-1}(A))$. 
\end{defn}
We note that $Z(A)=Z_{1}(A)$ since $[Z_{1}\left(A\right),A]$, $[A,Z_{1}(A)]\subseteq Z_{0}(A)=0$.
\begin{defn}
The coclass of $A$, denoted $cc(A)$, is given by $cc(A)=dim(A)-cl(A)$. 
\end{defn}
A known result for a Lie algebra $L$ is that if $dim\left(L\right)=n$,
then $dim\left(Z\left(L\right)\right)\neq n-1$\textbf{ }(\cite{holmesthesis},
Lemma 5). This result does not hold in Leibniz algebras since we do
not require $\left[a,a\right]=0$ for any element $a$. We can make
the following statement instead. 
\begin{lem}
Suppose $A$ is nilpotent, $dim\left(A\right)=n$ and $dim\left(Z\left(A\right)\right)=n-1$.
Then $A=I\oplus J$, where $I$ is the ideal with basis $\left\{ a,a^{2}\right\} $
for some $0\neq a\in A$ and $a^{2}\in Z\left(A\right)$, and $J$
is the ideal with the same basis elements as $Z\left(A\right)$ without
$a^{2}$.
\end{lem}
\begin{proof}
Let $a\in A$ but $a\notin Z(A)$. Then $a^{2}\neq0$, as otherwise
it would be in $Z(A)$. Then $I=span\{a,a^{2}\}$ with $a^{2}\in Z(A)$.
Take complementary subspace $J$ of $a^{2}\in Z(A)$, and the statement
follows. 
\end{proof}
There are numerous Leibniz results that are exactly the same as the
Lie results, including the proofs. These results can be found in (\cite{holmesthesis}).
The Frattini subalgebra is denoted by $\phi(A)$. One well known result
is that for a nilpotent Leibniz algebra $A$, $\phi\left(A\right)=\left[A,A\right]$
and it is the smallest ideal such that $A/\phi\left(A\right)$ is
abelian. Another result is that if $A$ is a nilpotent algebra with
P2 and $cl\left(A\right)=c$, then $Z_{c-1}\left(A\right)\subseteq\phi\left(A\right)$.
\begin{prop}
\label{prop:p2propzcminus1equalfratt}Suppose $A$ is nilpotent and
has P2. If $cl\left(A\right)=c$, then $Z_{c-1}\left(A\right)=\phi\left(A\right)$.
\end{prop}
\begin{proof}
By known result, $Z_{c-1}(A)\subseteq\phi(A)$. By definition, $A/Z_{c-1}(A)=Z_{c}(A)/Z_{c-1}(A)=\linebreak Z(A/Z_{c-1}(A))$
is abelian. Now, $\phi(A)\subseteq Z_{c-1}(A)$, since $\phi(A)$
is the smallest subalgebra which gives an abelian quotient algebra.
Hence, $Z_{c-1}(A)=\phi(A)$. 
\end{proof}
\begin{lem}
\label{lem:Acyclicifonlyfrattcodimone}Suppose $dim\left(A\right)>1$.
Then $A$ is cyclic if and only if the Frattini subalgebra has codimension
1 in $A$.
\end{lem}
\begin{proof}
Since $A$ is nilpotent, $\phi\left(A\right)=\left[A,A\right]$. Suppose
$A$ is cyclic. Then the derived algebra has codimension 1, and hence
the Frattini subalgebra has codimension 1. Conversely, suppose the
Frattini subalgebra is of codimension 1. Then $\phi\left(A\right)$
is the only maximal subalgebra. Let $a\in A$, such that $a\notin\phi\left(A\right)$.
The algebra it generates is contained in a maximal subalgebra or is
$A$. The former is not possible since $a\notin\phi\left(A\right)$.
Hence, $a$ generates $A$, and $A$ is cyclic. 
\end{proof}
\begin{lem}
\label{lem:Amodfrattdim2orAcyclic}Suppose $A$ is nilpotent and $dim\left(A\right)>1$.
Then $dim\left(A/A^{2}\right)=dim\left(A/\phi\left(A\right)\right)\geq2$
or $A$ is cyclic, and $dim\left(A/A^{2}\right)=1$.
\end{lem}
\begin{proof}
By Lemma (\ref{lem:Acyclicifonlyfrattcodimone}), $dim\left(A/A^{2}\right)=1$
if and only if $A$ is cyclic. Otherwise, $A$ has at least 2 maximal
subalgebras and their intersection has codimension 2 in $A$. The
Frattini subalgebra then has codimension greater than or equal to
2 in $A$.
\end{proof}
\begin{cor}
Suppose $A$ is nilpotent and has P2. Then $dim\left(A/Z_{c-1}\left(A\right)\right)\geq2$
or $A$ is cyclic.
\end{cor}
\begin{proof}
This is an immediate consequence of Proposition (\ref{prop:p2propzcminus1equalfratt})
and Lemma (\ref{lem:Amodfrattdim2orAcyclic}).
\end{proof}

\section{Coclasses 0 and 1}

In this section, we consider the algebras of coclasses 0 and 1. The
following lemma is easy to show.
\begin{lem}
Suppose $A$ is a nilpotent algebra with $N\trianglelefteq A$. The
following are true:
\end{lem}
\begin{enumerate}
\item If $dim\left(N\right)=s$, then $cc\left(A/N\right)\leq cc\left(A\right)$
\item If $N\subseteq Z\left(A\right)$ and $dim\left(N\right)>1$, then
$cc\left(A/N\right)\leq cc\left(A\right)-1$.
\end{enumerate}
For Lie algebras, it is a known result that if $dim\left(L\right)>2$
and $L$ has P2, then $dim\left(Z_{2}\left(L\right)\right)>2$ (\cite{holmesthesis},
Lemma 17). This result does not hold for Leibniz algebras, as shown
in the next example. The fact that this is a Leibniz algebra can be
seen in (\cite{fourdim}, Theorem 2.2). The lemma following this example
gives the alternative for Leibniz algebras.
\begin{example}
Let $A=span\left\{ x_{1},x_{2},x_{3},x_{4}\right\} $ with non-zero
multiplications given by $\left[x_{1},x_{1}\right]=x_{2}$, $\left[x_{1},x_{2}\right]=x_{3}$,
and $\left[x_{1},x_{3}\right]=x_{4}$. Since $A$ is cyclic, it has
P1, and so P2. Then $Z_{1}\left(A\right)=Z\left(A\right)=span\left\{ x_{4}\right\} $
and $Z_{2}\left(A\right)=span\left\{ x_{3},x_{4}\right\} $, and so
$dim\left(Z_{2}\left(A\right)\right)=2$.
\end{example}
\begin{lem}
\label{lem:cyclicleibsecondcenter}Suppose nilpotent $A$ has $P2$.
If $dim\left(A\right)\leq2$, then $A$ is cyclic or abelian. If $dim\left(A\right)>2$,
then $A$ is cyclic, $dim\left(Leib\left(A\right)\right)=1$, or $dim\left(Z_{2}\left(A\right)\right)>2$.
\end{lem}
\begin{proof}
We may assume that $dim(A)>2$ and that $A$ is not cyclic. If $dim(Leib(A))=0$,
then $A$ is Lie and $dim(Z_{2}(A))>2$ by (\cite{Karenpaper}, Lemma
6). Suppose $dim\left(Leib\left(A\right)\right)>1$. We will show
that $dim\left(Z_{2}\left(A\right)\right)>2$. Suppose that $dim\left(Z_{2}\left(A\right)\right)=2$.
Then $dim\left(Z\left(A\right)\right)=1$. Since $A$ is nilpotent,
$Z\left(A\right)\cap Leib\left(A\right)=Z\left(A\right)$ and $Z_{2}(A)\cap Leib(A)=Z_{2}(A)$.
Now $\left[Leib\left(A\right),A\right]=0$ always holds. Thus, for
any $x\in Z_{2}\left(A\right)$, $x\notin Z(A)$, $[A,x]\neq0$. The
kernel, $M$, of $R_{x}$, is shown to be a subalgebra of $A$. Hence,
it is an ideal since it has codimension 1 in the nilpotent $A$. Also,
$x\in Z_{2}\left(A\right)\subset Leib\left(A\right)$ and $x^{2}=0$,
so $x\in M$. Then $dim\left(Z\left(M\right)\right)\geq2$ as $M$
contains both $x$ and $Z\left(A\right)$.

Let $N$ be another maximal subalgebra of $A$ such that $N\neq M$.
Then $N$ is an ideal of $A$. Hence, $Z\left(A\right)\cap N\neq0$.
Therefore, $Z\left(A\right)\subset N$ and $Z\left(A\right)\subset Z\left(N\right)$.
Also, $dim\left(Z\left(N\right)\right)=dim\left(Z\left(M\right)\right)\geq2$.
Therefore, $Z_{2}\left(A\right)\cap Z\left(N\right)=Z_{2}\left(A\right)$
and $Z_{2}\left(A\right)\subset Z\left(N\right)$. So $R_{x}\left(a\right)=\left[a,x\right]=0$
for all $a\in N$. Thus, $N=M$, a contradiction.
\end{proof}
\begin{prop}
Suppose $cc\left(A\right)=0$. Then $A$ is cyclic, or $dim\left(A\right)\leq1$.
\end{prop}
\begin{proof}
If $A^{2}=\phi(A)$ has codimension 1 in $A$, it is cyclic. Since
$cc(A)=0$, the result follows. 
\end{proof}
\begin{thm}
\label{thm:ccone}Let $A$ be a nilpotent Leibniz algebra that satisfies
P2 and is of coclass 1. Then one of the following holds:

1.) $A$ is a Lie algebra, and so $A$ is abelian of dimension 2,
or $A$ is the Heisenberg Lie algebra of dimension 3

2.) $A=Z_{2}\left(A\right)$ and $dim\left(A\right)=3$. If $A=span\left\{ x,y,z\right\} $,
then $\left[x,x\right]=z$, $\left[y,y\right]=\tau z$, $\left[x,y\right]=\lambda z$,
$\left[y,x\right]=\varepsilon z$, where $\tau\neq0$ and $\left(\lambda+\varepsilon\right)^{2}-4$
is not a square. 
\end{thm}
\begin{proof}
If $A$ is a Lie algebra, then the result holds by (\cite{Karenpaper},
Proposition 3). $A$ is not cyclic since $cc(A)=1$. By Lemma (\ref{lem:cyclicleibsecondcenter}),
$dim(Leib(A))=1$ or $dim(Z_{2}(A))>2$. We consider each case.

Case 1: Suppose $dim\left(Z_{2}\left(A\right)\right)>2$. If $dim\left(Z_{2}\left(A\right)\right)\geq4$,
then $cc\left(A\right)\geq2$. Hence, $dim(Z_{2}(A))\linebreak=3$.
Then $Z_{2}\left(A\right)=A$ since $A$ is not cyclic and the next
to the last term in the upper central series has codimension greater
than 1 in A. Therefore, $dim\left(Z_{2}\left(A\right)\right)=3$ and
$A=Z_{2}\left(A\right)$. This also implies $dim\left(Z\left(A\right)\right)=1$.
Since $A$ is not Lie, $\left[A,A\right]=Leib\left(A\right)=Z\left(A\right)$.

Suppose $A=span\left\{ x,y,z\right\} $ with non-zero squares given
by one of

\hspace{.5cm} a.) $x^{2}=z$

\hspace{.5cm} b.) $x^{2}=z$, $y^{2}=\tau z$

\noindent and $Z\left(A\right)=span\left\{ z\right\} $. One maximal
subalgebra is $M_{1}=span\left\{ x,z\right\} $ and another is $M_{2}=span\left\{ y,z\right\} $.
Since they must be isomorphic, $\tau\neq0$ and $A$ satisfies (b).
Therefore, the algebra satisfies the multiplication in 2 in the statement
of the theorem. 

$M_{1}$ is cyclic, so any maximal subalgebra is also cyclic. Hence,
$M_{3}=span\{\alpha x+\beta y,z\}$ with $\alpha\neq0$ or $\beta\neq0$
must have 
\begin{align*}
0 & \neq\left(\alpha x+\beta y\right)^{2}\\
 & =\alpha^{2}\left[x,x\right]+\alpha\beta\left[x,y\right]+\alpha\beta\left[y,x\right]+\beta^{2}\left[y,y\right]\\
 & =\alpha^{2}z+\alpha\beta\lambda z+\alpha\beta\varepsilon z+\beta^{2}\tau z.
\end{align*}
We may take $\beta=1$. Consider $\alpha^{2}+\alpha\left(\lambda+\varepsilon\right)+\tau$.
$A$ satisfies P2 if and only if this expression is not 0 for any
$\alpha$, which is equivalent to $\left(\lambda+\varepsilon\right)^{2}-4\tau$
not being a square in $\mathbb{F}$. 

Case 2: Suppose that $dim\left(Leib\left(A\right)\right)=1$. Then
$A/Leib\left(A\right)$ is a Lie algebra of coclass 0 or 1 and satisfies
P2. If $A/Leib\left(A\right)$ has coclass 0, then $dim\left(A/Leib\left(A\right)\right)\leq1$.
If $A/Leib\left(A\right)$ has dimension 0, then $A=Leib\left(A\right)$
which is impossible. If $dim\left(A/Leib\left(A\right)\right)=1$,
then $dim\left(A\right)=2$ and $A$ is cyclic. Then $cc\left(A\right)=0$,
which is a contradiction.

Suppose $cc\left(A/Leib\left(A\right)\right)=1$. Then $A/Leib\left(A\right)$
is 2-dimensional abelian or 3-dimensional Heisenberg. In the first
case. $dim\left(A\right)=3$, $dim\left(Z\left(A\right)\right)=1$,
$Z_{2}\left(A\right)=A$, and $Z\left(A\right)=Leib\left(A\right)$.
This is the algebra considered in the last case.

Suppose $A/Leib\left(A\right)$ is Heisenberg. The next to last term
in the upper central series of $A$ has codimension greater than 1.
Therefore, $dim\left(Z\left(A\right)\right)=1$, $dim\left(Z_{2}\left(A\right)\right)=2$,
and $Z_{3}\left(A\right)=A$ is 4-dimensional. Then $Z\left(A\right)=Leib\left(A\right)$.
Hence, $A=span\left\{ w,x,y,z\right\} $ with $Z\left(A\right)=span\left\{ z\right\} $
and $Z_{2}\left(A\right)=span\left\{ y,z\right\} $.

The multiplication table for $A$ is 

\begin{table}[H]
\caption{Multiplications in $A$}

\begin{center}

\begin{tabular}{|c|c|c|c|c|}
\hline 
$\left[\cdot,\cdot\right]$ & $w$ & $x$ & $y$ & $z$\tabularnewline
\hline 
$w$ & $\alpha z$ & $y+az$ & $bz$ & $0$\tabularnewline
\hline 
$x$ & $-y+\hat{a}z$ & $\beta z$ & $cz$ & $0$\tabularnewline
\hline 
$y$ & $\hat{b}z$ & $\hat{c}z$ & $\gamma z$ & $0$\tabularnewline
\hline 
$z$ & $0$ & $0$ & $0$ & $0$\tabularnewline
\hline 
\end{tabular}

\end{center}
\end{table}
\noindent The Leibniz identity shows that $\hat{b}=-b$, $\hat{c}=-c$
and $\gamma=0$. There must be a non-zero square, so with a change
of basis, if necessary, we may assume that $\alpha\neq0$. Comparing
$M_{1}=span\left\{ w,y,z\right\} $ and $M_{2}=span\left\{ x,y,z\right\} $,
we have that $\beta\neq0$. Tables for $M_{1}$, $M_{2}$, and $M_{3}=span\left\{ mw+nx,y,z\right\} $
are

\begin{multicols}{2}

\begin{table}[H]
\caption{$M_{1}=span\left\{ w,y,z\right\} $}

\begin{center}

\begin{tabular}{|c|c|c|c|}
\hline 
$\left[\cdot,\cdot\right]$ & $w$ & $y$ & $z$\tabularnewline
\hline 
$w$ & $\alpha z$ & $bz$ & $0$\tabularnewline
\hline 
$y$ & $-bz$ & $0$ & $0$\tabularnewline
\hline 
$z$ & $0$ & $0$ & $0$\tabularnewline
\hline 
\end{tabular}

\end{center}
\end{table}

\columnbreak

\begin{table}[H]
\caption{$M_{2}=span\left\{ x,y,z\right\} $}

\begin{center}

\begin{tabular}{|c|c|c|c|}
\hline 
$\left[\cdot,\cdot\right]$ & $x$ & $y$ & $z$\tabularnewline
\hline 
$x$ & $\beta z$ & $cz$ & $0$\tabularnewline
\hline 
$y$ & $-cz$ & $0$ & $0$\tabularnewline
\hline 
$z$ & 0 & $0$ & $0$\tabularnewline
\hline 
\end{tabular}

\end{center}
\end{table}

\end{multicols}

\begin{table}[H]
\caption{$M_{3}=span\left\{ mw+nx,y,z\right\} $}

\begin{center}

\begin{tabular}{|c|c|c|c|}
\hline 
$\left[\cdot,\cdot\right]$ & $mw+nx$ & $y$ & $z$\tabularnewline
\hline 
$mw+nx$ & $\left(m^{2}\alpha+mna+mn\hat{a}+n^{2}\beta\right)z$ & $\left(mb+nc\right)z$ & $0$\tabularnewline
\hline 
$y$ & $-\left(mb+nc\right)z$ & $0$ & $0$\tabularnewline
\hline 
$z$ & $0$ & $0$ & $0$\tabularnewline
\hline 
\end{tabular}

\end{center}
\end{table}
\noindent If $b=0$, then $c=0$ since the center of $M_{1}$ and
$M_{2}$ have the same dimension. Then $cc\left(A\right)=2$, a contradiction.
Hence, $b\neq0\neq c$. Then $mb+nc\neq0$. But, we can find $m$
and $n$ such that $mb+nc=0$. Thus, $A$ does not satisfy P2 in this
case. 
\end{proof}

\section{Coclass 2}

In this section, we limit the scope of our investigation to the field
$\mathbb{C}$. If $A$ has coclass 2, then one of the following holds:
\begin{enumerate}
\item $dim(Z_{2}(A))=4$ and $Z_{2}(A)=A$,
\item $dim(Z_{2}(A))=3$, or
\item $dim(Z_{2}(A))=2$ and $dim(Leib(A))=1$. 
\end{enumerate}
Since $cc(A)=2$, $dim(Z_{2}(A))\leq4$. If $dim(Z_{2}(A))=4$, then
(1) holds since $dim(Z_{c-1})$ has codimension greater than or equal
to 2 in $A$. Otherwise, Lemma (\ref{lem:cyclicleibsecondcenter})
gives that (2) or (3) holds. Each of these cases is considered in
the following sections. We summarize the results in the following
theorem. 
\begin{thm}
\label{thm:cctwo}The non-Lie nilpotent Leibniz algebras with P1 over
$\mathbb{C}$ of coclass 2 are as follows:
\begin{enumerate}
\item If $A$ is split, then $A=span\left\{ x_{1},x_{2},x_{3},x_{4}\right\} $
with multiplications $\left[x_{1},x_{1}\right]=x_{3}$ and $\left[x_{2},x_{2}\right]=x_{4}$
\item If $A$ is non-split and $dim\left(A\right)=4$, then $A=span\left\{ x_{1},x_{2},x_{3},x_{4}\right\} $,
with multiplications given by one of the following: 
\begin{enumerate}
\item $[x_{1},x_{1}]=x_{3},[x_{2},x_{1}]=x_{4},[x_{1},x_{2}]=\alpha x_{3},[x_{2},x_{2}]=-x_{4},\alpha\in\mathbb{C}\backslash\left\{ -1\right\} $
\item $[x_{1},x_{1}]=x_{3},[x_{1},x_{2}]=x_{3},[x_{2},x_{1}]=x_{3}+x_{4},[x_{2},x_{2}]=x_{4}$.
\end{enumerate}
\item If $dim(A)=6$, then $A$ is given by one of the following: 
\begin{enumerate}
\item $A=span\{t,u,w,\hat{r}x,\hat{s}y,z\}$, with multiplications given
by $[t,u]=w=-[u,t]$, $[t,w]=x=-[w,t]$, $[t,\hat{r}x]=\hat{r}cz=-[\hat{r}x,t]$,
$[t,\hat{s}y]=\hat{s}dz$, $[\hat{s}y,t]=\hat{r}cz$, $[u,w]=y=-[w,u]$,
$[u,\hat{r}x]=\hat{r}fz$, $[u,\hat{s}y]=\hat{s}gz=-[\hat{s}y,u]$,
$[w,w]=\gamma z$, with the restrictions that $\hat{r},\hat{s},c,g,f,d,\hat{d},\gamma\neq0$,
$\hat{s}\hat{d}=\hat{r}c$, and $-d\neq\hat{d}$
\item $A=span\left\{ t,u,w,x,y,z\right\} $, with multiplications given
by $\left[t,u\right]=w=-\left[u,t\right]$, $\left[t,w\right]=x=-\left[w,t\right]$,
$\left[u,w\right]=y=-\left[w,u\right]$, $\left[w,w\right]=\gamma z$,
$\left[t,y\right]=dz$, $\left[y,t\right]=\hat{d}z$, $\left[u,x\right]=fz$,
$\left[x,u\right]=\hat{f}z$, with the restrictions that $2\gamma=d+\hat{d}=-f-\hat{f}$,
$-f=d$, and $-\hat{f}=\hat{d}$, where $\gamma,d,\hat{d},f,\hat{f}\in\mathbb{C}$.
\end{enumerate}
\end{enumerate}
\end{thm}
We collect several facts before starting the three cases for $Z_{2}(A)$.
An algebra is called split if it is the direct sum of non-zero ideals.
A nilpotent Leibniz algebra that is split has center of dimension
greater than 1. A result that we often use is the following lemma. 
\begin{lem}
\label{lem:AmodZP1}A nilpotent Leibniz algebra that satisfies P1
with ideal $B$ contained in $\phi(A)$ has that $A/B$ satisfies
P1. In particular, this holds for any term in the upper central series
of $A$.
\end{lem}

\subsection{$Dim(Z_{2}(A))=3$}

In this case, $A/Z_{2}(A)$ has coclass 1 and satisfies P1. Hence,
it is one of the algebras in Theorem (\ref{thm:ccone}). We consider
each of them.

\subsubsection{$A/Z_{2}(A)$ is Abelian}

Suppose that $A/Z_{2}(A)$ is abelian. Then $dim(A)=5$ and $dim(Z(A))=1$
or $2$. We claim that $A$ is non-split. If $dim(Z(A))=1$, then
the result is clear. If $dim(Z(A))=2$ and $A$ is the direct sum
of ideals $I$ and $J$, we may assume that $I$ has dimension 1 or
2. Since $A^{2}=Z(A)$, $I^{2}=Z(I)$. Clearly, $dim(I)$ can not
be one. If $dim(I)=2$, then $dim(Z_{2}(A))>3$ , a contradiction.
Hence $A$ is non-split.
\begin{prop}
There are no nilpotent non-Lie Leibniz algebras with P1 of coclass
2 where $dim\left(Z_{2}\left(A\right)\right)=3$ and $A/Z_{2}\left(A\right)$
is an abelian Lie algebra over $\mathbb{C}$. 
\end{prop}
\begin{proof}
Since $dim\left(Z_{2}\left(A\right)\right)=3$, we have that $cc\left(A/Z_{2}\left(A\right)\right)=1$.
By Theorem (\ref{thm:ccone}), we get that $A/Z_{2}\left(A\right)$
has dimension 2. It follows immediately that $dim\left(Z_{3}\left(A\right)\right)=dim\left(A\right)=5$,
$dim\left(Z_{2}\left(A\right)\right)=3$, and $dim\left(Z\left(A\right)\right)=1$
or $2$. By the notes preceeding this proposition, $A$ is non-split.
Since we have a 5-dimensional non-split Leibniz algebra over $\mathbb{C}$,
we may use (\cite{demir5dimclass}) to determine possible algebras.
By Section (\ref{sec:P1determinations}), the only algebras with P1
are $\mathcal{A}_{137}$, $\mathcal{A}_{138}\left(\alpha\right)$,
and $\mathcal{A}_{139}$. However, examining these algebras, we see
that for all of them, the center is given by $span\left\{ x_{3},x_{4},x_{5}\right\} $
and the second center is given by the entire 5-dimensional algebra.
Hence, these algebras are coclass 3. So there are no algebras satisfying
the given conditions.
\end{proof}

\subsubsection{$A/Z_{2}(A)$ is Heisenberg}

In this case, $dim(A)=6$, $dim(Z_{3}(A))=4$ and $dim(Z_{2}(A))=3$.
If $dim(Z(A))=2$, then $dim(A/Z(A))=4$, $cc(A/Z(A))=1$ and $A/Z(A)$
has P1. By Theorem (\ref{thm:ccone}), there are no possible algebras
in this case. Suppose that $dim(Z(A))=1$. Then $dim(A/Z(A))=5$,
then by Theorem 4 of (\cite{Karenpaper}), $A$ has basis $\{t,u,w,x,y,z\}$
with multiplication table

\begin{table}[H]
\caption{Multiplications in $A$}

\begin{center}

\begin{tabular}{|c|c|c|c|c|c|c|}
\hline 
$\left[\cdot,\cdot\right]$ & $t$ & $u$ & $w$ & $x$ & $y$ & $z$\tabularnewline
\hline 
$t$ & $\alpha z$ & $w+az$ & $x+bz$ & $cz$ & $dz$ & 0\tabularnewline
\hline 
$u$ & $-w+\hat{a}z$ & $\beta z$ & $y+ez$ & $fz$ & $gz$ & 0\tabularnewline
\hline 
$w$ & $-x+\hat{b}z$ & $-y+\hat{e}z$ & $\gamma z$ & $hz$ & $jz$ & 0\tabularnewline
\hline 
$x$ & $\hat{c}z$ & $\hat{f}z$ & $\hat{h}z$ & $\mu z$ & $kz$ & 0\tabularnewline
\hline 
$y$ & $\hat{d}z$ & $\hat{g}z$ & $\hat{j}z$ & $\hat{k}z$ & $\sigma z$ & 0\tabularnewline
\hline 
$z$ & 0 & 0 & 0 & 0 & 0 & 0\tabularnewline
\hline 
\end{tabular}

\end{center}
\end{table}

\noindent Using the Leibniz identity on all possible combinations
of basis elements, and then a simple change of basis, the multiplication
table can be taken to be

\begin{table}[H]
\caption{\label{tab:updmult}Updated Multiplications in $A$}

\begin{center}

\begin{tabular}{|c|c|c|c|c|c|c|}
\hline 
$\left[\cdot,\cdot\right]$ & $t$ & $u$ & $w$ & $x$ & $y$ & $z$\tabularnewline
\hline 
$t$ & $\alpha z$ & $w$ & $x$ & $cz$ & $dz$ & 0\tabularnewline
\hline 
$u$ & $-w+\bar{a}z$ & $\beta z$ & $y$ & $fz$ & $gz$ & 0\tabularnewline
\hline 
$w$ & $-x$ & $-y$ & $\gamma z$ & 0 & 0 & 0\tabularnewline
\hline 
$x$ & $-cz$ & $\hat{f}z$ & 0 & 0 & $0$ & 0\tabularnewline
\hline 
$y$ & $\hat{d}z$ & $-gz$ & 0 & 0 & 0 & 0\tabularnewline
\hline 
$z$ & 0 & 0 & 0 & 0 & 0 & 0\tabularnewline
\hline 
\end{tabular}

\end{center}
\end{table}
\noindent Note that we also get the following relationships from
calculating all possible Leibniz identities we get:

\begin{align}
\gamma & =d-f\\
\gamma & =-d-\hat{f}\\
\gamma & =\hat{d}+f
\end{align}
We consider the maximal subalgebras of $A$. Let $M=span\{mt+nu,w,x,y,z\}$
where not both $m$ and $n$ are 0. Special cases are $M_{1}$ when
$m=1$ and $n=0$, and $M_{2}$ where $m=0$ and $n=1$. 

Consider $M_{1}$ and $M_{2}$ and suppose $\alpha\neq0$. This means
$\beta\neq0$. Consider $M=span\{mt+nu,w,x,y,z\}$. Then $\left[mt+nu,mt+nu\right]=(m^{2}\alpha+mn\bar{a}+n^{2}\beta)z$.
Then choosing 
\[
m=\dfrac{-n\bar{a}\pm\sqrt{n^{2}\bar{a}^{2}-4\alpha\beta n^{2}}}{2\alpha}
\]
shows that $M$ is not isomorphic to $M_{1}$ or $M_{2}$. Hence,
we must have that $\alpha=0=\beta$. Then $0=[mt+nu,mt+nu]=mn\bar{a}$
implies that $\bar{a}=0$. 

From here, we continue to work on restrictions of the constants by
ensuring that $A$ has P1. Consider maximal subalgebra $M_{1}=span\left\{ t,w,x,y,z\right\} $
and $v=rx+sy\in Z\left(M_{1}\right)$. Then 

\begin{align*}
0 & =\left[t,rx+sy\right]=(rc+sd)z\\
0 & =\left[rx+sy,t\right]=(-rc+s\hat{d})z
\end{align*}
Adding the two equations together gives $0=s\left(d+\hat{d}\right)z$.
So either $s=0$ and/or $-d=\hat{d}$, as the results are not mutually
exclusive. If $s=0$, then $r=0$ and/or $c=0$., as again the results
are not mutually exclusive. This means if $s=0$, there are three
options for combinations of $r$ and $c$. So pairing the options
for $s=0$ with $-d=\hat{d}$, there are 7 total possibilities. Similarly,
using $M_{2}=span\left\{ u,w,x,y,z\right\} $ and $v'=r'x+s'y\in Z\left(M_{1}\right)$
\begin{align*}
0 & =[u,r'x+s'y]=(r'f+s'g)z\\
0 & =[r'x+s'y,u]=(r'\hat{f}-s'g)z
\end{align*}
and adding these equations gives $0=r'(f+\hat{f})$. So either $r'=0$
and/or $-f=\hat{f}$. If $r'=0$, then $s'=0$ and/or $g=0$. As before,
there are 7 possibilities. As options for $M_{1}$ and $M_{2}$ must
occur together, this gives $7\cdot7=49$ total possibilities concerning
$s,r,c,d+\hat{d},s',r',g$, and $f+\hat{f}$. We begin to eliminate
possibilities. 

First $r,s,r'$, and $s'$ concern the dimension of the center of
$M_{1}$ and $M_{2}$, which need to be the same for $A$ to have
P1. For example, if both of $r=0=s$, then if either $r'\neq0$ or
$s'\neq0$, that option may be immediately eliminated. There are 24
such cases that are immediately eliminated due to such conflicts between
$r,s,r'$, and $s'$. 

Next, consider the case where $s=r=s'=r'=0$, $c=0$, $g\neq0$, $-f\neq\hat{f}$,
and $-d\neq\hat{d}$. As $s=0=r$ implies that the $dim(Z(M_{1}))=1$.
However, if $c=0$, $x\in Z(M_{1})$, so $dim(Z(M_{1}))\geq2$, which
is a contradiction. There are 12 total cases that may now be immediately
eliminated due to contradictions between $r,s,$ and $c$, or between
$r',s',$ and $g$. 

Now, we turn to the possibilities where $-d\neq\hat{d}$ and $-f=\hat{f}$
or $-d=\hat{d}$ and $-f\neq\hat{f}$, and $c$ and $g$ are either
both 0 or both non-zero. In the first case, we have that $M_{1}$
is not skew-symmetric, while $M_{2}$ is skew-symmetric. The reverse
is true for the latter case. There are 4 new cases that may be eliminated
for this reason.

Consider the case when $s=0=c$, $-f=\hat{f}$, $-d\neq\hat{d}$,
and $r,r'\neq0$. In this case, $Z(M_{1})=span\{x,z\}\subseteq Z(M_{1})$.
However, $Z(M_{2})=span\{y,z\}\nsubseteq Z(M_{2})$. This case can
be eliminated. 

Similarly, we eliminate the case where $s,s'\neq0$, $r'=0$, $g=0$,
$-d=\hat{d}$, $-f\neq\hat{f}$. Here, $M_{1}^{2}=span\{x,z\}\nsubseteq Z(M_{1})$
but $M_{2}^{2}=span\{y,z\}\subseteq Z(M_{2})$. 

We now turn our attention towards the more interesting cases.

\

\noindent \textbf{Case 1:} 5 remaining cases where $-d=\hat{d}$
and $-f=\hat{f}$. 

We make use of equations (2), (3), and (4) above involving $\gamma$.
Subtracting the third from the second gives $-d-\hat{d}=\hat{f}+f$.
Adding the first to second and the first to third gives $2\gamma=d+\hat{d}=-f-\hat{f}$.
Considering these equations, $\gamma=0$. Since $\alpha=\beta=\bar{a}=0$,
we now have that all multiplications in $A$ are skew-symmetric and
$Leib\left(A\right)=\left\{ 0\right\} $. Hence $A$ is Lie, and is
given in (\cite{Karenpaper}, Theorem 4). 

\

\noindent \textbf{Case 2:} $s,r,s',r'=0$, $c,g\neq0$, $-f\neq\hat{f}$,
and $-d\neq\hat{d}$. 

In this case, $dim(Z(M_{1}))=1=dim(Z(M_{2}))$. Our goal now is to
find a maximal subalgebra that has center with dimension greater than
one, or to find a resulting algebra. Consider $M=span\left\{ mt+nu,w,x,y,z\right\} $.
By Table (\ref{tab:updmult}), a center element must be of the form
$\hat{r}x+\hat{s}y\in Z\left(M\right)$, where at least one of $\hat{r}$
or $\hat{s}$ are non-zero. Using a change of basis, we take $\hat{M}=span\{mt+nu,w,\hat{r}x+\hat{s}y,\hat{r}x-\hat{s}y,z\}$.
Again, by Table (\ref{tab:updmult}), it is clear that $\hat{r}x+\hat{s}y\in Z(M)$
if

\begin{equation}
0=\left[mt+nu,\hat{r}x+\hat{s}y\right]=(m\hat{r}c+m\hat{s}d+n\hat{r}f+n\hat{s}g)z\label{eq:centeq1}
\end{equation}
and 

\begin{equation}
0=\left[\hat{r}x+\hat{s}y,mt+nu\right]=(-m\hat{r}c+n\hat{r}\hat{f}+m\hat{s}\hat{d}-n\hat{s}g)z.\label{eq:centeq2}
\end{equation}
The goal is to find $m,n,\hat{r}$, and $\hat{s}$ that make $dim(Z(M))\geq2$.
As we only need one counter-example, we can limit our search of possible
subalgebras to those where $m,n,\hat{r},$ and $\hat{s}$ are all
non-zero. 

\

\noindent \textbf{Subcase 2.1: }$\hat{s}\hat{d}\neq\hat{r}c$ and
$\hat{s}g\neq\hat{r}\hat{f}$

Since equations (\ref{eq:centeq1}) and (\ref{eq:centeq2}) must hold
if $\hat{r}x+\hat{s}y\in Z(M)$, we can use them to find the necessary
constants $m,n,\hat{r}$, and $\hat{s}$. Since both equations are
constants in front of the basis element $z$, we can work with the
constants only. Begin by rearranging the constants in equation (\ref{eq:centeq2})
to get 
\begin{equation}
m\hat{r}c+n\hat{s}g=m\hat{s}\hat{d}+n\hat{r}\hat{f}.\label{eq:restrict}
\end{equation}
From here, we may replace $m\hat{r}c+n\hat{s}g$ in equation (\ref{eq:centeq1})
with $m\hat{s}\hat{d}+n\hat{r}\hat{f}$, which results in 
\begin{equation}
0=m\hat{s}\hat{d}+n\hat{r}\hat{f}+m\hat{s}d+n\hat{r}f\label{eq:updeq}
\end{equation}
subject to $m\hat{r}c+n\hat{s}g=m\hat{s}\hat{d}+n\hat{r}\hat{f}.$
Solving for $m$ in equations (\ref{eq:restrict}) and (\ref{eq:updeq}),
we get
\[
m=\dfrac{-n\hat{r}(f+\hat{f})}{\hat{s}(d+\hat{d})}
\]
and 
\[
m=\dfrac{n(\hat{s}g-\hat{r}\hat{f})}{\hat{s}\hat{d}-\hat{r}c}
\]
provided that $\hat{s}\neq0$, $-d\neq\hat{d}$, and $\hat{s}\hat{d}\neq\hat{r}c$.
We have already chosen $\hat{s}\neq0$, and we know $-d\neq\hat{d}$.
The case where $\hat{s}\hat{d}=\hat{\begin{gathered}r\end{gathered}
}c$ is handled below. From this point on, for this case, we assume $\hat{s}\hat{d}\neq\hat{r}c$.
This means we need 
\[
\dfrac{-n\hat{r}(f+\hat{f})}{\hat{s}(d+\hat{d})}=\dfrac{n(\hat{s}g-\hat{r}\hat{f})}{\hat{s}\hat{d}-\hat{r}c}
\]
so that $m$ holds. As $n$ is in both of these terms, and it is non-zero,
it need not be considered. As $m$ is assumed to be non-zero, we need
both of the above terms to be non-zero. The one on the left is 0 only
if $-f=\hat{f}$, which is not the case. The right hand side is 0
only if $\hat{s}g=\hat{r}\hat{f}$. The case where $\hat{s}g=\hat{r}\hat{f}$
is handled below. From this point on, for this case, we assume $\hat{s}g\neq\hat{r}\hat{f}$.
Now cross multiplying and canceling the $n$ yields
\[
-\hat{r}(f+\hat{f})(\hat{s}\hat{d}-\hat{r}c)=\hat{s}(\hat{s}g-\hat{r}\hat{f})(d+\hat{d})
\]
expanding gives

\[
-\hat{r}f\hat{s}\hat{d}+\hat{r}^{2}fc-\hat{r}\hat{f}\hat{s}\hat{d}+\hat{r}^{2}\hat{f}c=\hat{s}^{2}gd+\hat{s}^{2}g\hat{d}-\hat{s}\hat{r}\hat{f}d-\hat{s}\hat{r}\hat{f}\hat{d}
\]
and cancelling like terms gives
\[
-\hat{r}f\hat{s}\hat{d}+\hat{r}^{2}fc+\hat{r}^{2}\hat{f}c=\hat{s}^{2}gd+\hat{s}^{2}g\hat{d}-\hat{s}\hat{r}\hat{f}d
\]
and moving everything to one side 
\[
0=-\hat{r}f\hat{s}\hat{d}+\hat{r}^{2}fc+\hat{r}^{2}\hat{f}c-\hat{s}^{2}gd-\hat{s}^{2}g\hat{d}+\hat{s}\hat{r}\hat{f}d
\]
and rearranging gives 

\[
0=\hat{r}^{2}c(f+\hat{f})+\hat{r}\hat{s}(d\hat{f}-\hat{d}f)-\hat{s}^{2}g(d+\hat{d})
\]
Via the quadratic equation in terms of $\hat{r}$, we have 
\[
\hat{r}=\dfrac{-\hat{s}(d\hat{f}-\hat{d}f)}{2c(f+\hat{f})}\pm\dfrac{\sqrt{\hat{s}^{2}(d\hat{f}-\hat{d}f)^{2}+4c(f+\hat{f})\hat{s}^{2}g(d+\hat{d})}}{2c(f+\hat{f})}
\]
and we know that $c\neq0$ and $-f\neq\hat{f}$. Hence, as we are
over $\mathbb{C}$, we can solve for $\hat{r}$ in terms of ${\displaystyle \hat{s}}$.
This means we have that 
\[
m=\dfrac{-n\hat{r}(f+\hat{f})}{\hat{s}(d+\hat{d})}=\dfrac{n(\hat{s}g-\hat{r}\hat{f})}{\hat{s}\hat{d}-\hat{r}c}\neq0
\]
for the appropriate choice of $\hat{r}$ and $\hat{s}$. Upon inspection
of various scenarios, the only concern with the quadratic equation
being 0 is if both $d\hat{f}-\hat{d}f=0$ and $g=0$, but we know
$g\neq0$, and so the quadratic gives a valid non-zero solution for
$\hat{r}$. 

Provided $\hat{s}\hat{d}\neq\hat{r}c$ and that $\hat{s}g\neq\hat{r}\hat{f}$,
we summarize what has been shown. We needed to find $m,n,\hat{r},$
and $\hat{s}$ so that equations (\ref{eq:centeq1}) and (\ref{eq:centeq2})
both held. What we have shown is that we can find $m\neq0$ by taking
\[
m=\dfrac{-n\hat{r}(f+\hat{f})}{\hat{s}(d+\hat{d})}=\dfrac{n(\hat{s}g-\hat{r}\hat{f})}{\hat{s}\hat{d}-\hat{r}c}
\]
for any choice of $n\neq0$. We have shown that 
\[
\dfrac{-n\hat{r}(f+\hat{f})}{\hat{s}(d+\hat{d})}=\dfrac{n(\hat{s}g-\hat{r}\hat{f})}{\hat{s}\hat{d}-\hat{r}c}
\]
for the given solution of $\hat{r}$ as $\hat{s}\neq0$. This means
we can find $m,n,\hat{r}$, and $\hat{s}$ which force equations (\ref{eq:centeq1})
and (\ref{eq:centeq2}) to hold. Therefore, $dim(Z(\hat{M}))\geq2$.
Thus, $A$ would not have P1 under these conditions, and so we cannot
have $s=0=r$, unless $\hat{s}\hat{d}=\hat{r}c$ or $\hat{s}g=\hat{r}\hat{f}$.

\

\noindent \textbf{Subcase 2.2:} $\hat{s}\hat{d}=\hat{r}c$ and $\hat{s}g=\hat{r}\hat{f}$.

If both of these equations hold, then 
\[
\dfrac{\hat{s}\hat{d}}{c}=\hat{r}=\dfrac{\hat{s}g}{\hat{f}}.
\]
Note that $\hat{d},\hat{f}\neq0$, as otherwise these equalities would
not hold given the restrictions on the other variables. For simplicity,
take $\hat{s}=1$. Take 
\[
m=\dfrac{-ng(\hat{f}+f)}{\hat{f}(d+\hat{d})}.
\]
which exists and is non-zero. Then 
\begin{align*}
m\hat{r}c+m\hat{s}d+n\hat{r}f+n\hat{s}g & =\dfrac{-ng(\hat{f}+f)}{\hat{f}(d+\hat{d})}\cdot\dfrac{\hat{d}}{c}\cdot c+\dfrac{-ng(\hat{f}+f)}{\hat{f}(d+\hat{d})}\cdot d+n\cdot\dfrac{g}{\hat{f}}\cdot f+ng\\
 & =(\dfrac{-ng(\hat{f}+f)}{\hat{f}(d+\hat{d})})\cdot(d+\hat{d})+\dfrac{ngf}{\hat{f}}+\dfrac{ng\hat{f}}{\hat{f}}\\
 & =\dfrac{-ng(\hat{f}+f)}{\hat{f}}+\dfrac{ng(f+\hat{f})}{\hat{f}}\\
 & =0
\end{align*}
and 
\begin{align*}
-m\hat{r}c+n\hat{r}\hat{f}+m\hat{s}\hat{d}-n\hat{s}g & =\dfrac{ng(\hat{f}+f)}{\hat{f}(d+\hat{d})}\cdot\dfrac{\hat{d}}{c}\cdot c+n\cdot\dfrac{g}{\hat{f}}\cdot\hat{f}+\dfrac{-ng(\hat{f}+f)}{\hat{f}(d+\hat{d})}\cdot\hat{d}-ng\\
 & =\dfrac{ng(\hat{f}+f)}{\hat{f}(d+\hat{d})}\cdot\hat{d}+ng+\dfrac{-ng(\hat{f}+f)}{\hat{f}(d+\hat{d})}\cdot\hat{d}-ng\\
 & =0.
\end{align*}

This shows that equations (\ref{eq:centeq1}) and (\ref{eq:centeq2})
are both 0, and we have found a maximal subalgebra with $dim(Z(M))=2$.
In this case, $A$ does not have P1.

\

\noindent \textbf{Subcase 2.3: }Only $\hat{s}\hat{d}=\hat{r}c$

Note that $\hat{d}\neq0$, as otherwise we would be in Subcase 2.1.
Since $\hat{f}\neq-f$ and $\hat{d}\neq-d$, we know $\gamma\neq0$.
Now, for $M_{1}=span\{t,w,\hat{r}x+\hat{s}y,\hat{r}x-\hat{s}y,z\}$,
we have $Z^{l}(M_{1})=span\{\hat{r}x+\hat{s}y,z\}$. Then the only
way for $M_{2}=span\{u,w,\hat{r}x+\hat{s}y,\hat{r}x,z\}$ to have
$dim(Z^{l}(M_{2}))=2$ is if $\hat{f}=0$. Consider $M=span\{mt+nu,w,\hat{r}x+\hat{s}y,\hat{r}x,z\}$,
which is isomorphic to $M_{1}$ and $M_{2}$, provided a second left
center element exists. Say this element is of the form $ht+ju+pw+q\hat{r}x+k\hat{s}y$.
If it is the left center then 
\begin{align*}
0 & =[ht+ju+pw+q\hat{r}x+k\hat{s}y,mt+nu]\\
 & =hnw-jmw-pmx-pny-q\hat{r}mcz+k\hat{s}m\hat{d}z-k\hat{s}ngz\\
 & =(hn-jm)w-pmx-pny+(-q\hat{r}mc+k\hat{s}m\hat{d}-k\hat{s}ng)z.
\end{align*}
This implies that we choose $h$ and $j$ so that $hn-jm=0$, which
is possible over $\mathbb{C}$. We get that $p=0$. Lastly, $-q\hat{r}mc+k\hat{s}m\hat{d}-k\hat{s}ng=0$
gives $k\hat{s}(m\hat{d}-ng)-qm\hat{s}\hat{d}=0$, using the fact
that $\hat{s}\hat{d}=\hat{r}c$. Recall that $\hat{s},\hat{d}\neq0$.
If $m\hat{d}-ng=0$, then $q=0$. If $k=0$, then $q=0$. Otherwise,
\[
k=\dfrac{qm\hat{d}}{(m\hat{d}-ng)},
\]
which again is no problem over $\mathbb{C}$. Hence, for any general
maximal subalgebra $M$, we can find a second left center element,
and therefore all maximal subalgebras are isomorphic. We get the following
final algebra:

\begin{table}[H]
\caption{\label{tab:A1finalmult}Final Multiplications in $A_{1}$}

\begin{center}

\begin{tabular}{|c|c|c|c|c|c|c|}
\hline 
$\left[\cdot,\cdot\right]$ & $t$ & $u$ & $w$ & $\hat{r}x$ & $\hat{s}y$ & $z$\tabularnewline
\hline 
$t$ & 0 & $w$ & $x$ & $\hat{r}cz$ & $\hat{s}dz$ & 0\tabularnewline
\hline 
$u$ & $-w$ & 0 & $y$ & $\hat{r}fz$ & $\hat{s}gz$ & 0\tabularnewline
\hline 
$w$ & $-x$ & $-y$ & $\gamma z$ & 0 & 0 & 0\tabularnewline
\hline 
$\hat{r}x$ & $-\hat{r}cz$ & $0$ & 0 & 0 & $0$ & 0\tabularnewline
\hline 
$\hat{s}y$ & $\hat{r}cz$ & $-\hat{s}gz$ & 0 & 0 & 0 & 0\tabularnewline
\hline 
$z$ & 0 & 0 & 0 & 0 & 0 & 0\tabularnewline
\hline 
\end{tabular}

\end{center}
\end{table}

\noindent with the restrictions that $\hat{r},\hat{s},c,g,f,d,\hat{d},\gamma\neq0$,
$\hat{s}\hat{d}=\hat{r}c$, and $-d\neq\hat{d}$. The fact that $d\neq0$
comes from the fact that $0\neq\gamma=-d-\hat{f}$ and $\hat{f}=0$. 

\

\noindent \textbf{Subcase 2.4: }Only $\hat{s}g=\hat{r}\hat{f}$

Note that $\hat{f}\neq0$, as otherwise we would be in Subcase 2.1.
We still have $\hat{f}\neq-f$ and $\hat{d}\neq-d$ implying $\gamma\neq0$.
This time, if you consider the left centers of $M_{1}=span\{t,w,\hat{r}x+\hat{s}y,\hat{s}y,z\}$
and $M_{2}=span\{u,w,\hat{r}x+\hat{s}y,\hat{r}x-\hat{s}y,z\}$ we
get that $\hat{d}=0$. The same process can be used as in Subcase
2.3 to show that a general maximal subalgebra $M=span\{mt+nu,w,\hat{r}x+\hat{s}y,\hat{s}y,z\}$
can be made to have a second left center element, and all maximal
subalgebras are isomorphic. We get the following algebra

\begin{table}[H]
\caption{Final Multiplications in $A_{2}$}

\begin{center}

\begin{tabular}{|c|c|c|c|c|c|c|}
\hline 
$\left[\cdot,\cdot\right]$ & $t$ & $u$ & $w$ & $\hat{r}x$ & $\hat{s}y$ & $z$\tabularnewline
\hline 
$t$ & 0 & $w$ & $x$ & $\hat{r}cz$ & $\hat{s}dz$ & 0\tabularnewline
\hline 
$u$ & $-w$ & 0 & $y$ & $\hat{r}fz$ & $\hat{sgz}$ & 0\tabularnewline
\hline 
$w$ & $-x$ & $-y$ & $\gamma z$ & 0 & 0 & 0\tabularnewline
\hline 
$\hat{r}x$ & $-\hat{r}cz$ & $\hat{s}gz$ & 0 & 0 & $0$ & 0\tabularnewline
\hline 
$\hat{s}y$ & 0 & $-\hat{s}gz$ & 0 & 0 & 0 & 0\tabularnewline
\hline 
$z$ & 0 & 0 & 0 & 0 & 0 & 0\tabularnewline
\hline 
\end{tabular}

\end{center}
\end{table}

\noindent with the restrictions that $\hat{r},\hat{s},c,g,d,f,\hat{f},\gamma\neq0$,
$\hat{s}g=\hat{r}\hat{f}$, and $-f\neq\hat{f}$. Again, we make use
of $0\neq\gamma=\hat{d}+f$. Swapping the rows and columns for $t$
and $u$, and swapping the rows and columns for $\hat{r}x$ and $\hat{s}y$,
the tables for $A_{1}$ and $A_{2}$ are the same, so $A_{1}\simeq A_{2}$. 

\

\noindent \textbf{Case 3:} $s=r'=c=g=0$, $r,s'\neq0$, $\hat{f}\neq-f$
and $\hat{d}\neq-d$

Again, $\gamma\neq0$.Use the new basis for $M$ given by $M=span\left\{ mt+nu,w,mx+ny,y,z\right\} $,
where $m\neq0$. Then $M_{1}^{3}=0$, so $M^{3}=0$. Therefore, $mn(d+f)=0$
and $mn(\hat{d}+\hat{f})=0$. Replace $y$ by $(1/m)y$ in the basis
for $M$ and the resulting multiplication shows that $M$ and $M_{1}$
are isomorphic. A similar procedure works if we assume $n\neq0$.
So all maximal subalgebras are isomorphic, $A$ satisfies P1, and
has coclass 2. The final description of the algebra has multiplication
table

\begin{table}[H]
\caption{\label{tab:A3finalmult}Final Multiplications in $A_{3}$}

\begin{center}

\begin{tabular}{|c|c|c|c|c|c|c|}
\hline 
$\left[\cdot,\cdot\right]$ & $t$ & $u$ & $w$ & $x$ & $y$ & $z$\tabularnewline
\hline 
$t$ & 0 & $w$ & $x$ & 0 & $dz$ & 0\tabularnewline
\hline 
$u$ & $-w$ & 0 & $y$ & $fz$ & 0 & 0\tabularnewline
\hline 
$w$ & $-x$ & $-y$ & $\gamma z$ & 0 & 0 & 0\tabularnewline
\hline 
$x$ & 0 & $\hat{f}z$ & 0 & 0 & $0$ & 0\tabularnewline
\hline 
$y$ & $\hat{d}z$ & 0 & 0 & 0 & 0 & 0\tabularnewline
\hline 
$z$ & 0 & 0 & 0 & 0 & 0 & 0\tabularnewline
\hline 
\end{tabular}

\end{center}
\end{table}

\noindent with the restrictions that $2\gamma=d+\hat{d}=-f-\hat{f}$,
$-f=d$, and $-\hat{f}=\hat{d}.$

\

It remains to consider the case where $A/Z\left(A\right)$ is a 5-dimensional
Leibniz algebra. In this case, we still know that $A/Z_{2}\left(A\right)$
is the 3-dimensional Heisenberg Lie algebra and $dim\left(Z\left(A\right)\right)=1$.
This implies the upper central series of $B=A/Z\left(A\right)$ is
$dim\left(Z\left(B\right)\right)=2$, $dim\left(Z_{2}\left(B\right)\right)=3$,
and $dim\left(B\right)=5$. Since $A$ has P1, $B$ will have P1 by
Lemma (\ref{lem:AmodZP1}), so we know that $\phi\left(A\right)=\left[A,A\right]=Z_{2}\left(B\right)$.
The only possible Leibniz algebras fitting these requirements and
having P1 are $\mathcal{A}_{137}$, $\mathbb{\mathcal{A}}_{138}\left(\alpha\right)$,
and $\mathcal{A}_{139}$ in Subsection (\ref{sec:P1determinations}),
which were coclass 3. Hence, there are no possibilities for this case. 
\begin{prop}
Suppose that $A$ is a nilpotent Leibniz algebra over $\mathbb{C}$
where $A/Z_{2}(A)$ is the three-dimensional Heisenberg Lie algebra,
with multiplications given by table (\ref{tab:A1finalmult}) or (\ref{tab:A3finalmult}).
Then $A$ has 2 coclass 2 and P1.
\end{prop}

\subsubsection{$A/Z_{2}(A)$ is a non-Lie Leibniz Algebra}

The last possibility is that $A/Z_{2}\left(A\right)$ is a non-Lie
Leibniz algebra in Theorem (\ref{thm:ccone}). However, as we are
working over the complex numbers, the conditions in the proposition
will not be satisfied. Hence, there are no possible algebras for this
case. 

\subsection{$Dim(Leib(A))=1$}

The case $dim(Z_{2}(A))=3$ has been considered in the preceding work
and $dim(Z_{2}(A))=4$ will be considered in the next section, so
we take $dim(Z_{2}(A))=2$. Therefore $dim(Z(A))=1$, $Z(A)=Leib(A)$,
and $A/Leib(A)$ has coclass 2. The next result lists the possible
Lie algebras for $A/Leib(A)$, and we consider each case. 
\begin{thm}
\label{thm:liecoclass2}(\cite{holmesthesis}, Theorem 4) If $dim(L)=n$,
$cc(L)=2$, and L has P1, then $L$ is isomorphic to one of the following
algebras: 
\begin{enumerate}
\item[(i)]  $\left\langle \left\langle a,b,c\right\rangle \right\rangle $ where
$\left[a,b\right]=\left[a,c\right]=\left[b,c\right]=0$
\item[(ii)] $\left\langle \left\langle x,y,z,a,b\right\rangle \right\rangle $
where $\left[x,y\right]=z$, $\left[x,z\right]=a$, $\left[y,z\right]=b$
\item[(iii)]  $\left\langle \left\langle a,b,c,x,y,z\right\rangle \right\rangle $
where $\left[a,b\right]=c$, $\left[a,c\right]=x$, $\left[b,c\right]=y$,
$\left[a,x\right]=z$, $\left[b,y\right]=\gamma z$ where $-\gamma$
is not a perfect square
\end{enumerate}
\end{thm}
If $A/Leib(A)$ is abelian, then $A=Z_{2}(A)$ has dimension 4, which
is not the case we are currently considering. In the second case in
Theorem (\ref{thm:liecoclass2}), $A/Leib(A)$ has dimension 5 and
$dim(Z_{2}(A))=3$. This case was considered in the last section.
In the third case in Theorem (\ref{thm:liecoclass2}), $\gamma$ is
not a perfect square, which cannot happen since we are considering
algebras over the complex numbers. Hence there are no new algebras
from this section.

\subsection{$A=Z_{2}(A)$with $dim(A)=4$}

Now $A^{2}=Z(A)$ by Proposition (\ref{prop:p2propzcminus1equalfratt}).
If $A$ is split, then $A$ is the direct sum of ideals. If $I$ is
one of them and $dim(I)=1$, then $Z(I)=I$ and $I^{2}=0$, which
is not possible. Hence $I$ is the direct sum of two ideals of dimension
2, each of whose center equals its derived algebra. Hence each is
cyclic. Then any maximal subalgebra is the direct sum of a two-dimensional
cyclic ideal and a one-dimensional ideal. Hence $A$ has P1 and is
of coclass 2. This algebra is listed in Theorem (\ref{thm:cctwo}).

Suppose that $A$ is not split. Then $A^{2}=Z(A)$ has dimension 1
or 2. The algebras in (\cite{fourdim}) with these conditions are
checked to see if they have P1 in Section (\ref{sec:P1determinations}).
Two algebras are found and are listed in Theorem (\ref{thm:cctwo})
under the non-Lie algebras of dimension 4. 

\section{Determinations of Leibniz Algebras that have P1 \label{sec:P1determinations}}

In the sections above, we refer to known non-split 4- and 5-dimensional
Leibniz algebras which have P1. These non-split algebras are classified
in (\cite{fourdim}) and (\cite{demir5dimclass}). In order to determine
which algebras to check, we make use of coclass, the dimension of
the center, the dimension of $Leib(A)$, and $dim(Z_{n-1}(A))$. Notably,
the classifications of algebras in (\cite{fourdim}) and (\cite{demir5dimclass})
is done using the lower central series, while the work in this paper
uses the upper central series. However, for nilpotent Leibniz algebras,
the lengths of the upper and lower central series are the same, and
so the coclass is the same. Also, as our algebras are required to
have P1, we know that $A^{n-1}=[A,A]=Z_{n-1}(A)$, and the dimension
of the centers will be the same. For the 4-dimensional algebras in
(\cite{fourdim}), Theorems 2.1, 2.3, and 2.5 were checked. For the
5-dimensional algebras in (\cite{demir5dimclass}), Theorems 2.3,
2.4, 3,6, 3.7, 3.8, 3.9, 3.10, and 3.11 were checked. 

The determinations of whether these algebras have P1 is mostly done
outside of this paper. However, below we give several examples of
how Leibniz algebras were shown to not have P1. After that, we show
the proofs for those algebras that in fact have P1. The work for determining
which algebras have P1 can be found in (\cite{mydissertation}). 

The first way to eliminate algebras having P1 was to show that the
upper central series of two maximal subalgebras was not the same.
First, we consider $\mathcal{A}_{1}=span\{x_{1},x_{2},x_{3},x_{4}\}$
in Theorem 2.1 in (\cite{fourdim}). The multiplications are given
by $[x_{1},x_{3}]=x_{4}$ and $[x_{3},x_{2}]=x_{4}$. Take maximal
subalgebra $M_{1}=span\{x_{1},x_{3},x_{4}\}$, which has $Z(M_{1})=span\{x_{4}\}$,
and $M_{1}=Z_{2}(M_{1})$. Next, take maximal subalgebra $M_{2}=span\{x_{1},x_{2},x_{4}\}$,
which is abelian. Hence, $\mathcal{A}_{1}$ does not have P2, and
so does not have P1.

The next way to eliminate algebras having P1 was to show that $Leib(M_{1})$
and $Leib(M_{2})$ for two maximal subalgebras did not have the same
dimension. Notably, if $dim(Leib(M_{1}))\neq0$, but $dim(Leib(M_{2}))=0$,
then the second subalgebra would have been a Lie algebra, while the
first was not, indicating that the subalgebras are not isomorphic,
and so $A$ does not have P1. For example, take $\mathcal{A}_{8}=span\{x_{1},x_{2},x_{3},x_{4},x_{5}\}$
in Theorem 2.3 in (\cite{demir5dimclass}). This algebra is defined
by the multiplications $[x_{1},x_{1}]=x_{5}$, $[x_{1},x_{2}]=x_{3}=-[x_{2},x_{1}]$,
$[x_{1},x_{3}]=x_{4}=-[x_{3},x_{1}]$, and $[x_{2},x_{3}]=x_{5}=-[x_{3},x_{2}]$.
Consider maximal subalgebras $M1=span\{x_{1},x_{3},x_{4},x_{5}\}$
and $M2=span\{x_{2},x_{3},x_{4},x_{5}\}$. In $M_{1}$, the non-zero
multiplications are given by $[x_{1},x_{1}]=x_{5}$, $[x_{1},x_{3}]=x_{4}=-[x_{3},x_{1}]$,
so $Leib(M_{1})=span\{x_{5}\}$, and $M_{1}$ is not a Lie algebra.
In $M_{2}$, the only non-zero multiplications are given by $[x_{2},x_{3}]=x_{5}=-[x_{3},x_{2}]$,
and so $Leib(M_{2})=0$, which implies $M_{2}$ is Lie. Hence $M_{1}$
and $M_{2}$ are not isomorphic, and $\mathcal{A}_{8}$ does not have
P1.

The 4-dimensional non-split Leibniz algebras with P1 are given in
Theorem (\ref{thm:cctwo}) above. The proofs that these algebras have
P1 is given next. There are no 5-dimensional algebras with P1 of coclass
2.

\pagebreak
\begin{thm}
Algebras $\mathcal{A}_{18}$ and $\mathcal{A}_{19}$ in Theorem 2.5
in (\cite{fourdim}) have P1.
\end{thm}
\begin{proof}
The first algebra is given by $\mathcal{A}_{18}$: $[x_{1},x_{1}]=x_{3}$,
$[x_{2},x_{1}]=x_{4}$, $[x_{1},x_{2}]=\alpha x_{3}$, $[x_{2},x_{2}]=-x_{4}$,
$\alpha\in\mathbb{C}\backslash\left\{ -1\right\} $. All maximal subalgebras
are of the form $M=span\left\{ ax_{1}+bx_{2},x_{3},x_{4}\right\} $.
Now 
\begin{align*}
\left[ax_{1}+bx_{2},ax_{1}+bx_{2}\right] & =a^{2}x_{3}+ab\alpha x_{3}+abx_{4}-b^{2}x_{4}\\
 & =\left(a^{2}+ab\alpha\right)x_{3}+\left(ab-b^{2}\right)x_{4}.
\end{align*}
Change the basis for $M$, and let $r=ax_{1}+bx_{2}$ and $s=\left(a^{2}+ab\alpha\right)x_{3}+\left(ab-b^{2}\right)x_{4}.$
Choose $t$ to be complementary to $s$ in $\left\{ x_{3},x_{4}\right\} $.
Then all maximal subalgebras can be written as $M'=span\left\{ r,s,t\right\} $
and the only multiplication is $r^{2}=s$. As this holds for all maximal
subalgebras, $\mathcal{A}_{18}$ has P1. 

The second algebra is given by $\mathcal{A}_{19}$: $[x_{1},x_{1}]=x_{3}$,
$[x_{1},x_{2}]=x_{3}$, $[x_{2},x_{1}]=x_{3}+x_{4}$, $[x_{2},x_{2}]=x_{4}$.
All maximal subalgebras are of the form $M=span\left\{ ax_{1}+bx_{2},x_{3},x_{4}\right\} $.
Now 
\begin{align*}
\left[ax_{1}+bx_{2},ax_{1}+bx_{2}\right] & =a^{2}x_{3}+abx_{3}+abx_{3}+abx_{4}+b^{2}x_{4}\\
 & =\left(a^{2}+2ab\right)x_{3}+\left(ab+b^{2}\right)x_{4}.
\end{align*}
Change the basis for $M$, and let $r=ax_{1}+bx_{2}$ and $s=\left(a^{2}+2ab\right)x_{3}+\left(ab+b^{2}\right)x_{4}$.
Choose $t$ to be complementary to $s$ in $\left\{ x_{3},x_{4}\right\} $.
Then all maximal subalgebras can be written as $M'=span\left\{ r,s,t\right\} $
and the only multiplication is $r^{2}=s$. Since this holds for all
maximal subalgebras, $\mathcal{A}_{19}$ has P1. 
\end{proof}

\section*{Acknowledgements}

The author would like to thank Ernie Stitzinger for his encouragement
and suggestions.

\bibliographystyle{plain}
\bibliography{publicationbiblio}

\begin{thebibliography}{1}

\bibitem{demir5dimclass}
Ismail Demir.
\newblock Classification of 5-dimensional complex nilpotent leibniz algebras,
  June 2017.

\bibitem{mydissertation}
Lindsey Farris.
\newblock {\em Finite Dimensional Nilpotent Leibniz Algebras with Isomorphic
  Maximal Subalgebras}.
\newblock PhD thesis, North Carolina State University, March 2022.

\bibitem{Hermann}
P{\'e}ter~Z. Hermann.
\newblock On finite p-groups with isomorphic maximal subgroups.
\newblock {\em Journal of the Australian Mathematical Society}, 48(2):199--213,
  April 1990.

\bibitem{Karenpaper}
Karen Holmes and Ernest Stitzinger.
\newblock Finite dimensional nilpotent lie algebras with isomorphic maximal
  subalgebras.
\newblock {\em Communications in Algebra}, 29(6):2501--2521, 2001.

\bibitem{holmesthesis}
Karen Marie~Brown Holmes.
\newblock {\em Finite Dimensional Nilpotent Lie Algebras with Isomorphic
  Maximal Subalgebras}.
\newblock PhD thesis, North Carolina State University, 1999.

\bibitem{MannGpThry}
Avinoam Mann.
\newblock On $p$-groups whose maximal subgroups are isomorphic.
\newblock {\em Journal of the Australian Mathematical Society}, 59(2):143--147,
  1995.

\bibitem{stitz}
Ismail Demir{,} Kailash~C. Misra{,} and Ernie Stitzinger.
\newblock On some structures of leibniz algebras.
\newblock {\em Contemporary Mathematics}, 623:41--54, 2014.

\bibitem{fourdim}
Ismail Demir{,} Kailash C. Misra{,}~Ernest Stitzinger.
\newblock On classificiation of four-dimensional nilpotent leibniz algebras.
\newblock {\em Communications in Algebra}, 45(3):1012--1018, 2017.

\end{thebibliography}

\end{document}